\documentclass[a4wide,10pt]{article}
\usepackage{amsfonts,amsmath,amsthm,amssymb}
\usepackage{epsfig}
\usepackage{a4wide}
\newtheorem{theorem}{Theorem}

\newtheorem{lemma}[theorem]{Lemma}
\newcommand{\ds}{\displaystyle}
\newcommand{\ts}{\textstyle}

\title{On the probability of planarity of a \\ random graph  near the critical point}

\author{Marc Noy \thanks{Departament de Matem\`atica Aplicada II.
Universitat Polit\`ecnica de Catalunya.
Jordi Girona 1-3,   08034 Barcelona. Spain. {\tt
marc.noy@upc.edu}. Partially supported by grants MTM2011-24097 and DGR2009-SGR1040.} \and
Vlady Ravelomanana \thanks{Liafa Umr CNRS 7089. Universit\'e Denis Diderot.
175, Rue du Chevaleret 75013 Paris. France. {\tt vlad@liafa.jussieu.fr}.}
\and
Juanjo Ru\'e \thanks {Instituto de Ciencias Matem\'{a}ticas.
Calle Nicol\'{a}s Cabrera 15, 28049 Madrid. Spain. {\tt juanjo.rue@icmat.es}.
Partially supported by grants JAE-DOC (CSIC),  MTM2011-22851 and  SEV-2011-0087.}
}
\date{}

\begin{document}
\maketitle

\begin{abstract}
Let  $G(n,M)$ be the uniform random graph with $n$ vertices and $M$ edges.
Erd\H{o}s and R\'enyi (1960) conjectured that the limiting probability 
$$
\lim_{n \to  \infty} \Pr\{G(n,\textstyle{n\over 2}) \hbox{ is planar}\}
$$
exists and is a  constant strictly  between $0$ and $1$.
\L uczak, Pittel and Wierman (1994) proved this conjecture and
Janson, \L uczak, Knuth and Pittel (1993)
gave lower and upper bounds for this probability.
In this paper we determine the exact limiting probability of a random graph being planar near the critical point $M=n/2$.
For each  $\lambda$, we find an exact analytic expression for
$$ p(\lambda) =  \lim_{n \to  \infty} \Pr\left\{G\left(n,\textstyle{n\over 2}(1+\lambda n^{-1/3})\right) \hbox{ is planar} \right\}.$$
In particular, we obtain $p(0) \approx 0.99780$.
We extend these results to classes of graphs closed under taking minors. As an example, we show that the probability of $G(n,\textstyle{n\over 2})$ being series-parallel converges to $0.98003$.
For the sake of completeness and exposition we reprove in a concise way several basic properties we need of a random graph near the critical point.
\end{abstract}

\bigskip
\begin{flushright}
\textit{We dedicate this paper to the memory of Philippe Flajolet.}
\end{flushright}

\section{Introduction}

The random graph model $G(n,M)$ assigns uniform probability to graphs on $n$ labelled vertices with $M$ edges.
A fundamental result of  Erd\H os and R\'{e}nyi \cite{erdos} is that  the random graph $G(n,M)$ undergoes an abrupt change when $M$ is around $n/2$, the value for which  the average vertex degree is equal to one. When $M = cn/2$ and $c<1$, almost surely the connected components are all  of order $O(\log n)$, and are either trees or unicyclic graphs. When $M = cn/2$ and $c>1$, almost surely there is a unique giant component of size $\Theta(n)$.
We direct to reader to the  reference texts \cite{bollobas} and \cite{janson} for a detailed discussion of these facts.

We concentrate on the so-called critical window 
$M= {n\over 2} (1 +\lambda  n^{-1/3}),$
where $\lambda$ is a real number, identified by the work of Bollob\'{a}s \cite{bollobas84a,bollobas84b}.
Let us recall that the excess of a connected graph is the number of edges minus the number of vertices. A connected graph is complex if it has positive excess.
As $\lambda \to -\infty$, complex components  disappear and only trees and unicyclic components survive, and as $\lambda \to +\infty$, components with unbounded excess appear.
A thorough analysis of the random graph in the critical window can be found in \cite{giant} and~\cite{LuczakStruct}, which constitute our basic references.

For each fixed $\lambda$, we denote the random graph $G\left(n,\textstyle{n\over 2}(1+\lambda n^{-1/3})\right)$ by $G(\lambda)$.  The core $C(\lambda)$ of $G(\lambda)$ is obtained by repeatedly removing all vertices of degree one from $G(\lambda)$. The kernel $K(\lambda)$ is obtained from $C(\lambda)$ by replacing all maximal paths of vertices of degree two by single edges.
The graph $G(\lambda)$ satisfies almost surely several fundamental properties, that were established in~\cite{LuczakStruct} by a subtle simultaneous analysis of the $G(n,M)$ and the $G(n,p)$ models.

\begin{enumerate}
 \setlength{\itemsep}{-1pt}%
  \item The number of complex components is bounded.
  \item Each complex component has size of order $n^{2/3}$, and the largest suspended tree in each complex component has size of order $n^{2/3}$.
  \item $C(\lambda)$ has size of order $n^{1/3}$ and maximum degree three, and the distance between two vertices of degree three in $C(\lambda)$  is of order $n^{1/3}$.
     \item $K(\lambda)$ is a cubic (3-regular) multigraph of bounded size.
    \end{enumerate}
The key property for us is the last one.
It implies that almost surely the components of $G(\lambda)$ are trees, unicyclic graphs, and those obtained from a cubic multigraph~$K$ by attaching rooted trees to the vertices of $K$, and attaching ordered sequences of rooted trees to the edges of $K$. Some care is needed here, since the resulting graph may not be simple, but  asymptotically this can be accounted for.

It is clear that $G(\lambda)$ is planar if and only if the kernel $K(\lambda)$ is planar. Then by counting planar cubic multigraphs it is possible to estimate the probability that $G(\lambda)$ is planar. To this end we use generating functions. The trees attached to $K(\lambda)$ are encoded by the generating function~$T(z)$ of rooted trees, and complex analytic methods are used to estimated the coefficients of the corresponding series.
This allows us to determine the exact probability
$$ p(\lambda) =  \lim_{n \to  \infty} \Pr\left\{G\left(n,\textstyle{n\over 2}(1+\lambda n^{-1/3})\right) \hbox{ is planar} \right\}.$$
In particular, we obtain $p(0) \approx 0.99780$.

This approach was initiated in  the seminal paper by Flajolet, Knuth and Pittel \cite{FKP89}, where the authors determined the threshold for the appearance of the first cycles in $G(n,M)$.
A basic feature  in \cite{FKP89} is  to estimate coefficients of large powers of generating functions using Cauchy integrals and the saddle point method.
This path was followed  by Janson, Knuth, \L uczak and Pittel~\cite{giant}, obtaining a wealth of results on~$G(\lambda)$. Of particular importance for us is the determination in \cite{giant} of the limiting probability that $G(\lambda)$ has given excess.
The approach by \L uczak, Pittel and Wierman in~\cite{LuczakStruct} is more probabilistic and has as starting point the classical estimates by Wright \cite{wright} on the number of connected graphs with fixed excess.
The range of these estimates was extended by Bollob\'{a}s \cite{bollobas84a} and more recently the analysis was refined by Flajolet, Salvy and Schaeffer~\cite{airy}, by  giving complete asymptotic expansions in terms of the Airy function.

The paper is organized as follows. In Section \ref{sec:prelim} we present the basic lemmas needed in the sequel. Except for the proof of Lemma~\ref{le:kernel}, the paper is self-contained.
Lemmas \ref{le:multigraph} to
\ref{le:airy} are proved in \cite{giant} with a different presentation; for the sake of completeness and exposition  we provide shorter and hopefully more accessible  proofs.
In Section \ref{sec:cubic} we compute the number of cubic planar multigraphs, suitably weighted, where we follow \cite{Kang-Luczak}. 
In Section \ref{sec:main} we compute the exact probability that the random graph $G(\lambda)$ is planar as a function of $\lambda$. We generalize this result by determining the probability that $G(\lambda)$ belongs to a minor-closed class of graphs in several cases of interest.

We close this introduction with a remark. The problem of 2-satisfiability presents a striking  analogy with the random graph process. Given $n$ Boolean variables and a conjunctive formula of $M$ clauses, each involving two literals, the problem is to determine the probability that the formula is satisfiable when $M$ grows with $n$. The threshold has been established at $M=n$ and the critical window is also of width $n^{2/3}$; see  \cite{2-sat}. However the exact probability of satisfiability when the number of clauses is $n(1 + \lambda n^{-1/3})$ has not been determined, and appears to be a more difficult problem.

\section{Preliminaries}\label{sec:prelim}

All graphs in this paper are labelled.
The size of a graph is its number of vertices.
A multigraph is a graph with loops and multiple edges allowed.

We recall that the generating function $T(z)$ of rooted trees satisfies
$$
    T(z) = z e^{T(z)}.
$$
Using Lagrange's inversion \cite{FS}, one recovers the classical formula $n^{n-1}$ for the number of rooted labelled trees.The generating function for unrooted trees is
$$U(z) = T(z) - {T(z)^2\over2}.$$
This can be proved by integrating the relation $T(z) = z U'(z)$, or more combinatorially using the dissimilarity theorem for trees \cite{otter}.

A graph is unicyclic if it is connected and has a unique cycle.
Unicyclic graphs can be seen as an undirected cycle of length  at least three to which we attach a sequence of rooted trees. Since the directed cycle construction  corresponds algebraically to $\log(1/(1-T(z))$ (see \cite{FS}), the generating function is
$$
    V(z) = {1 \over 2} \left( \log{1 \over 1-T(z)} - T(z) - {T(z)^2 \over 2} \right).
    $$
Graphs all whose components are unicyclic are given by the exponential formula:
$$
    e^{V(z)} = {e^{-T(z)/2-T(z)^2/4} \over \sqrt{1-T(z)}}.
$$

The following result, which is fundamental for us, is proved in \cite[Theorem 4]{LuczakStruct} by a careful analysis of  the structure of complex components in $G(\lambda)$.
We say that a property $\cal P$ holds asymptotically almost surely (a.a.s.) in $G(n,M)$ if the probability that  $\cal P$ is satisfied tends to one as $n\to\infty$.
Recall that $G(\lambda) = G\left(n,\textstyle{n\over 2}(1+\lambda n^{-1/3})\right)$.

\begin{lemma}\label{le:kernel}
For each $\lambda$, the kernel of $G(\lambda)$ is a.a.s.  a cubic multigraph.
\end{lemma}

Given a cubic multigraph $M$ with $a$ loops, $b$ double edges and $c$ triple edges, define its weight~ as
$$
        w(M) = 2^{-a} 2^{-b} 6^{-c}.
$$
This weight (called the compensation factor in \cite{giant}), has the following explanation. When we substitute edges of the kernel by sequences of rooted trees, a loop has two possible orientations that give raise to the same graph. A double (triple) edge can be permuted in two (six) ways, again producing the same graph. From now on, all multigraphs we consider are weighted, so that we omit the qualifier.
The following lemma is proved in \cite{giant} using a combination of guessing and  recurrence relations. The next proof is already contained in
\cite[Chap. 2]{bollobas}.

\begin{lemma}\label{le:multigraph}
The number $E_r$ of cubic multigraphs with $2r$ vertices is equal to
$$
 E_r = {(6r)! \over  (3r)! 2^{3r} 6^{2r}}.
$$
\end{lemma}

\begin{proof}
A cubic multigraph can be modeled as a pairing of darts (half-edges), 3 for each vertex, with a total of $6r$ darts.
The number of such pairings is $(6r)! /((3r)! 2^{3r})$.
However, we have to divide by the number $6^{2r}$ of ways of permuting each of the $2r$ triples of darts. The weight takes care exactly of the number of times a cubic multigraph is produced in this process.
\end{proof}

The next result is essentially proved in \cite{giant}  using several algebraic manipulations. Here we present a concise proof. We denote by $[z^n]A(z)$ the coefficient of $z^n$ in the power series $A(z)$.

\begin{lemma}\label{le:counting}
The number $g(n,M,r)$ of simple graphs with $n$ vertices, $M$ edges and cubic kernel of size $2r$ satisfies
$$            g(n,M,r) \le n! \,[z^n] {U(z)^{n-M+ r} \over (n-M+r)!} \, e^{V(z)}  {E_r \over (2r)!} {T(z)^{2r} \over (1-T(z))^{3r}}             $$

            and
$$             g(n,M,r) \ge n! \, [z^n] {U(z)^{n-M+ r} \over (n-M+r)! } \, e^{V(z)} {E_r \over (2r)!} {T(z)^{8r} \over (1-T(z))^{3r}}.             $$

\end{lemma}

\begin{proof}
Such a graph is the union of a set of $s$ unrooted trees, a set of unicyclic
graphs, and a cubic multigraph $K$ with a rooted tree attached to each vertex of $K$ and a sequence (possibly empty) of rooted trees attached to each edge of $K$.
Let us see first that $s = n-M+r$. Indeed, the final excess of edges over vertices must be $M-n$. Each tree component contributes with excess~$-1$, each unicyclic component with excess $0$, and $K$ (together with the attached trees) with excess $r$. Hence
$ M-n = -s + r.$

The first two factors $U(z)^{n-M+ r}/(n-M+r)!$ and $ e^{V(z)}$ on the right-hand side  of the inequalities encode the set of trees and unicyclic components. The last part encodes the kernel $K$. It has~$2r$ vertices and is labelled, hence the  factor $E_r/(2r)!$;  the weighting guarantees that each graph contributing towards $g(n,M,r)$ is counted exactly once. The trees  attached to the $2r$ vertices  give a factor  $T(z)^{2r}$. The  sequences of trees attached to the $3r$ edges give each a factor $1/(1-T(z))$. However, this allows for the empty sequence and the resulting graph may not be simple, so we get only an upper bound. To guarantee that the final graph is simple we take  sequences of length at least two, encoded by $T(z)^2/(1-T(z))$ (length one is enough for multiple edges of $K$, but length two is needed for loops). Since this misses some graphs, we get a lower bound.
\end{proof}

The following technical result is essentially Lemma 3 from \cite{giant}.
 We reprove it here for completeness in a simplified version tailored to our needs (see also the proof of Theorem 5 in \cite{FKP89}).

\begin{lemma}\label{le:integral}
Let $M=\frac{n}{2}(1+\lambda n^{-1/3})$. Then for any fixed  $a$
 and integer $r>0$ we have
\begin{equation} \label{eq:estimate}
\frac{n!}{{ {n \choose 2} \choose M}}
[z^n] \frac{U(z)^{n-M+r}}{(n-M+r)!} \frac{T(z)^{a}}{(1-T(z))^{3r}} \,
e^{V(z)} = \sqrt{2 \pi} A\left(3r+\ts\frac{1}{2},\lambda\right)
 \left( 1 + O\left(\frac{\lambda^4}{n^{1/3}}\right)\right)
\end{equation}
uniformly for $|\lambda|\le n^{1/12}$, where
\begin{equation} \label{DEFAIRY}
A(y,\lambda)={e^{-\lambda^3\!/6}\over3^{(y+1)/3}}\sum_{k\ge0}
{\bigl({1 \over 2} 3^{2/3}\lambda\bigr)^k\over k!\,\Gamma\bigl((y+1-2k)/3\bigr)}.
\end{equation}
\end{lemma}

\begin{proof} The proof is based on relating the left-hand side of Equation~\eqref{eq:estimate} to the integral representation of $A(y,\lambda )$ defined in~\cite[Equation(10.7)]{giant}:
$$
A(y,\lambda)={1\over2\pi i}\int_{\Pi} s^{1-y}e^{K(\lambda,s)} ds,
$$
where $K(\lambda,s)$ is the polynomial
$$
K(\lambda,s)={(s+\lambda)^2(2s-\lambda)\over6}=
{s^3\over3}+{\lambda s^2\over2}-{\lambda^3\over6}
$$
and $\Pi$ is a path in the complex plane that consists of the following three straight line segments:
$$
s(t)=\left\{\begin{array}{rr}
        -e^{-\pi i/3}\,t,& \textrm{for} -\infty<t\le-2,\\
       1+it\sin\pi/3,& \textrm{for} -2\le t\le+2,\\
        e^{+\pi i/3}\,t, & \textrm{for} +2\le t<+\infty.
\end{array}\right.
$$
The constant term in the left-hand side of 
\eqref{eq:estimate} is estimated using Stirling's formula, getting
\begin{equation}\label{STIRLING}
\frac{n!}{{ {n \choose 2} \choose M}} \, \frac{1}{(n-M+r)!}
= \sqrt{2 \pi n} \,\frac{2^{n-M+r}}{n^r}  e^{-\lambda^3/6 + 3/4 - n}
 \left( 1+O\left(\frac{\lambda^4}{n^{1/3}}\right)\right).
\end{equation}
The coefficient of $[z^n]$ in Equation~\eqref{eq:estimate} is estimated by means of a contour integral around $z=0$, using the expressions of $U(z)$ and $V(z)$ in terms of $T(z)$:
$$
\frac{1}{2 \pi i} \oint \left(T(z) - \frac{T(z)^2}{2}\right)^{n-M+r}
\frac{T(z)^a \,  e^{-T(z)/2-T(z)^2/4} }{(1-T(z))^{3r+1/2}} \frac{dz}{z^{n+1}}.
$$
We make the change of variable $u=T(z)$, whose inverse is $z=ue^{-u}$, and we obtain
\begin{equation}\label{cauchy}
\frac{2^{M-n-r} e^{n}}{2 \pi i} \oint g(u) \, e^{n h(u)} \frac{du}{u},
\end{equation}
where the integrand is split into a smooth function
$$
g(u) = \frac{u^{a} (2u-u^2)^{r}  e^{-u/2-u^2/4}}{(1-u)^{3r-1/2}}
$$
and a large power involving
$$
h(u) = u-1-\log u-\left(1-{M\over n}\right)
\log{1\over 1-(u-1)^2}.
$$
The contour path in Equation~\eqref{cauchy} should be such that
$|u|<1$.  As remarked in~\cite{giant} (see also~\cite{FKP89}), the function $h(u)$ satisfies
$h(1)=h'(1)=0$. Moreover, precisely at the critical value
$M= n/2$ we also have $h''(1)=0$. This triple zero shows up
in the procedure  used in \cite{giant} when estimating (\ref{cauchy}) for large~$n$ by means of the saddle-point method.
Notice that $h(u)$ is singular at $u=1$, due to the singularity at $z=e^{-1}$ of $T(z)$.

Let $\nu=n^{-1/3}$, and let $\alpha$ be the positive solution to
$$
\lambda = \alpha^{-1} - \alpha.
$$
This choice is necessary in order to get precise bounds for the tail estimates that appear using the saddle-point method. Following the proof of~\cite[Lemma 3]{giant}, we evaluate (\ref{cauchy}) on the
path  $u=e^{-(\alpha+it)\nu}$,
where $t$ runs from $-\pi n^{1/3}$ to $\pi n^{1/3}$. That is,
$$
\oint f(u){du\over u}=i\nu \, \int_{-\pi n^{1/3}}^{\pi n^{1/3}}
f(e^{-(\alpha+it)\nu})\,dt.
$$
The main contribution to the value of this integral
comes from the vicinity of $t=0$. The magnitude of
$e^{h(u)}$ depends on the real part $\Re h(u)$. Observe that
 $\Re h(e^{-(\alpha+it)\nu})$ decreases as $|t|$
increases, and  that $|e^{nh(u)}|$ has its
 maximum on the circle $u=e^{-(\alpha+it)\nu}$
when $t=0$. We write $s=\alpha+it$. Analyzing  $nh(e^{-s\nu})$, we have
$$
n\,h(e^{-s\nu})=\frac{s^3} {3}+
\frac{\lambda s^2}{2}
+O\bigl((\lambda^2s^2+s^4)\nu
\bigr),
$$
uniformly in any region such that $|s\nu| < \log 2$. For the function $g(u)$, we have
$$
g\left(e^{-s\nu}\right) = \frac{\left(2e^{-s\nu} - e^{-2s\nu}\right)^r}{\left(1-e^{-s\nu}\right)^{3r-1/2}} \,
 e^{-s\nu a - e^{-s\nu}/2 - e^{-2s\nu}/4}
   = (s\nu)^{1/2-3r} e^{-3/4} \left( 1 + O(s\nu)\right).
$$
If $f(u)=g(u)e^{nh(u)}$ is the integrand of~(\ref{cauchy}), we have
$$
e^{-{\lambda^3/6}} f(e^{-s\nu}) = e^{-3/4} \nu^{1/2-3r} s^{1-(3r+1/2)} e^{K(\lambda,s)} \left(1+O(s\nu)+O(\lambda^2s^2\nu)+O(s^4\nu) \right)
$$
when $s=O(n^{1/12})$. Finally,
$$
\frac{e^{-\lambda^3/6}}{2\pi i} \oint f(u) \frac{du}{u} = e^{-3/4} \nu^{3/2- 3r} A(3r+\ts\frac{1}{2},\lambda) +
O(\nu^{5/2-3r} e^{-\lambda^3/6} \lambda^{3r/2+1/4}),
$$
where the error term has been derived from those already in~\cite{giant}. The proof of the lemma is completed by multiplying~(\ref{STIRLING}) and (\ref{cauchy}), and canceling equal terms.
\end{proof}

It is important to notice that in the previous lemma the final asymptotic estimate does \emph{not} depend on the choice of $a$. The next result is a direct consequence and can be found as Formula (13.17) in \cite{giant}.

\begin{lemma}\label{le:airy}
The limiting probability that the random graph $G(\lambda)$ has a cubic kernel of size $2r$ is equal to
$$
    \sqrt{2 \pi} \, e_r A(3r+\ts{1\over2}, \lambda),
$$
where $e_r = E_r /(2r)!$ and $A(y,\lambda)$ is as in the previous lemma.

In particular, for $\lambda=0$ the limiting probability is
$$
    \sqrt{2 \over 3} \left({4\over3}\right)^r e_r {r! \over (2r)!}.
$$
\end{lemma}

\begin{proof}
Using the notation of Lemma \ref{le:counting}, the probability for a given $n$ is by definition
$$
    {g(n,M,r) \over {{n \choose 2} \choose M} }.
$$
Lemma \ref{le:counting} gives upper and lower bounds for this probability, and
Using Lemma \ref{le:integral} we see that both bounds agree in the limit and are equal to
$$
    {E_r \over (2r)!}  \sqrt{2 \pi}  A(3r+\ts{1\over2},\lambda),
$$
thus proving the result.
A key point is  that the discrepancy between the factors $T(z)^{2r}$ and $T(z)^{5r}$ in the bounds for $g(n,M,r)$ does not affect the limiting value of the probability.
\end{proof}

Notice that if we replace the $e_r$ by the numbers $g_r$ arising by counting  \emph{planar} cubic multigraphs, we obtain immediately the probability that $G(\lambda)$ has a cubic planar kernel of size~$2r$. Since $G(\lambda)$ is planar if and only if its kernel is planar, we can use this fact to compute the probability of~$G(\lambda)$ being planar. But first we must compute $g_r$.

\section{Planar cubic multigraphs}\label{sec:cubic}

In this section we compute the numbers $G_r$ of cubic planar multigraphs of size $2r$. The associated generating function  has been obtained recently by Kang and \L uczak \cite{Kang-Luczak} (generalizing the enumeration of simple cubic graphs in \cite{cubic-planar}), but their derivation contains some minor errors. They do not affect the correctness of \cite{Kang-Luczak}, since the asymptotic estimates needed by the authors are still valid. However, for the computations that follow we need the exact values. The next result is from \cite{Kang-Luczak}, with the corrections mentioned below.
All multigraphs are weighted as in the previous section.

\begin{lemma}\label{le:cubic-planar}
Let $G_1(z)$ be the generating function of connected cubic planar multigraphs. Then~$G_1(z)$ is determined by the following system of equations:
$$
\renewcommand{\arraystretch}{1.5}
\begin{array}{lll}
3z \ds\frac{d G_1(z)}{dz}&=&D(z)+C(z)\\
B(z)&=& \ds\frac{z^2}{2}(D(z)+C(z))+\frac{z^2}{2}\\
C(z)&=& S(z)+P(z)+H(z)+B(z)\\
D(z)&=& \ds\frac{B(z)^2}{z^2}\\
S(z)&=& C(z)^2-C(z)S(z)\\
P(z)&=&z^2C(z)+\ds\frac{1}{2}z^2C(z)^2+\ds\frac{z^2}{2}\\
2(1+C(z))H(z) &=&  u(z)(1-2u(z))-u(z)(1-u(z))^3\\
z^2(C(z)+1)^3&=& u(z)(1-u(z))^3.
\end{array}
$$
\end{lemma}
The generating functions  $B(z),\, C(z),\, D(z),\, S(z),\, P(z)$ and $H(z)$ correspond to distinct families of edge-rooted cubic planar graphs, and $u(z)$ is an algebraic function related to the enumeration of 3-connected cubic planar graphs (dually, 3-connected triangulations).

The corrections with respect to \cite{Kang-Luczak} are the following. In the first equation a term $-7z^2/24$ has been removed. In the second and sixth equations we have replaced a term $z^2/4$ by $z^2/2$. In the fourth equation we have removed a term $-z^2/16$.
For the combinatorial interpretation of the various generating functions and the  proof of the former equations we refer to \cite{Kang-Luczak}.
Notice that eliminating $u(z)$ from the last two equations we obtain a relation between $C(z)$ and~$H(z)$. This relation can be used to obtain
a single equation satisfied by $C(z)$, eliminating  from the remaining equations. We  reproduce it here in case the reader wishes to check our computations.
$$
\begin{array}{ll}
1048576\,{z}^{6}+1034496\,{z}^{4}-55296\,{z}^{2}+ \left( 9437184\,{z}^
{6}+6731264\,{z}^{4}-1677312\,{z}^{2}+55296 \right) C+ \\
\left( 37748736
\,{z}^{6}+18925312\,{z}^{4}-7913472\,{z}^{2}+470016 \right) {C}^{2}+ \\
 \left( 88080384\,{z}^{6}+30127104\,{z}^{4}-16687104\,{z}^{2}+1622016
 \right) {C}^{3}+ \\
 \left( 132120576\,{z}^{6}+29935360\,{z}^{4}-19138560
\,{z}^{2}+2928640 \right) {C}^{4}+ \\ \left( 132120576\,{z}^{6}+19314176
\,{z}^{4}-12429312\,{z}^{2}+2981888 \right) {C}^{5}+ \\
\left( 88080384\,
{z}^{6}+8112384\,{z}^{4}-4300800\,{z}^{2}+1720320 \right) {C}^{6}+ \\
 \left( 37748736\,{z}^{6}+2097152\,{z}^{4}-614400\,{z}^{2}+524288
 \right) {C}^{7}+ \\ \left( 9437184\,{z}^{6}+262144\,{z}^{4}+65536
 \right) {C}^{8}+1048576\,{C}^{9}{z}^{6} = 0.
 \end{array}
$$
The first terms are
$$
    C(z) = {z}^{2}+{\frac {25}{8}}\,{z}^{4}+{\frac {59}{4}}\,{z}^{6}+{\frac {
11339}{128}}\,{z}^{8}+ \cdots
$$
This allows us to compute $B(z),D(z), S(z), P(z)$ and $H(z)$, hence also $G_1(z)$.
The first coefficients of $G_1(z)$ are as follows.
$$
G_1(z) = \frac{5}{24}{z}^{2}+{\frac {5}{16}}{z}^{4}+{\frac
{121}{128}}{z}^{6}+{\frac {1591}{384}}{z}^{8}+\cdots
$$
Using the set construction, the generating function $G(z)$ for cubic planar multigraphs is then
\begin{equation} \label{coeff-planar}
G(z) = e^{G_1(z)} = \sum_{r=0}^{\infty} G_r {z^{2r} \over (2r)!} =
1+{\frac {5}{24}}{z}^{2}+{\frac {385}{1152}}{z}^{4}+{\frac {83933}{
82944}}{z}^{6}+{\frac {35002561}{7962624}}{z}^{8} + \cdots,
\end{equation}
where $G_r$ is the  number of planar cubic multigraphs with $2r$ vertices.
This coincides with the generating function for all cubic (non-necessarily planar) multigraphs up to the coefficient of~$z^4$.
The first discrepancy is in the coefficient of $z^6$. The difference between the coefficients is $1/72=10/6!$, corresponding to the 10 possible ways of labelling  $K_{3,3}$, the unique non-planar cubic multigraph on six vertices.

\section{Probability of planarity and generalizations}\label{sec:main}

Let $G$ be a graph with a cubic kernel $K$. Then  clearly  $G$ is planar if and only if $K$ is planar, and we can compute the probability that $G(n,M)$ is planar by counting over all possible planar kernels.

\begin{theorem}\label{th:planar}
Let $g_r(2r)!$ be the number of cubic planar multigraphs with $2r$ vertices.
Then the limiting probability that the random graph $G(n,M={n \over 2}(1+\lambda n^{-1/3}))$ is planar is
$$
             p(\lambda) =  \sum_{r \ge 0} \sqrt{2 \pi} \, g_r A(3r+\ts{1\over2},\lambda).
$$
In particular, the limiting probability that $G(n, {n \over 2})$ is planar is
$$
    p(0) =  \sum_{r \ge 0} \sqrt{2 \over 3} \left(\frac{4}{3}\right)^r g_r {r! \over (2r)!} \approx
    0.99780.
$$
\end{theorem}

\begin{proof}
The same analysis as in Section \ref{sec:prelim} shows that
$\sqrt{2 \pi} \, g_r A(3r+\ts{1\over2},\lambda)$ is the probability that the kernel is planar and has $2r$ vertices. Summing over all possible $r$, we get the desired result.
\end{proof}

As already mentioned, Erd\H{o}s and R\'enyi \cite{erdos} conjectured that $p(0)$ exists and $0 < p(0) < 1$. This was proved in \cite{LuczakStruct}, showing  that  $p(\lambda)$ exists for all $\lambda$ and that $0 < p(\lambda) < 1$.
The bounds in \cite{giant} for $p(0)$ are
$$
            0.98707 < p(0) < 0.99977,
$$
obtained by considering connected cubic multigraphs with at most six vertices.
We remark that Stepanov \cite{stepanov} showed that $p(\lambda) < 1$ for $\lambda \le 0$ (without actually establishing the existence of the limiting probability).
The function $p(\lambda)$ is plotted in Figure 1.
As expected, $p(\lambda)$ is close to 1 when $\lambda\to -\infty$ and close
to 0 when $\lambda\to \infty$. For instance, $p(-3) \approx 1- 1.02 \cdot 10^{-7}$ and $p(5) \approx 4.9 \cdot 10^{-7}$.

Besides planar graphs, one can consider other classes of graphs.
Let $\mathcal{G}$ be a class of graphs closed under taking minors, that is, if $H$ is a minor of $G$ and $G\in\mathcal{G}$, then $H \in \mathcal{G}$.
If $H_1,\cdots, H_k$ are the excluded minors of $\mathcal{G}$, then we write $\mathcal{G} = \hbox{Ex}(H_1,\dots,H_k)$. (By the celebrated theorem of Robertson and Seymour, the number of excluded minors is finite, but we do not need this deep result here).
The following result  generalizes the previous theorem.

\begin{theorem}
Let $\mathcal{G}=\hbox{Ex}(H_1,\dots,H_k)$ and assume all the $H_i$ are 3-connected.
Let  $h_r(2r)!$ be the number of cubic multigraphs in $\mathcal{G}$ with $2r$ vertices.
Then the limiting probability that the random graph $G(n,M={n \over 2}(1+\lambda n^{-1/3}))$ is in $\mathcal{G}$ is
$$
             p_{\mathcal{G}}(\lambda) =  \sum_{r \ge 0} \sqrt{2 \pi} \, h_r A(3r+\ts{1\over2},\lambda).
$$
In particular, the limiting probability that $G(n, {n \over 2})$ is in $\mathcal{G}$ is
$$
    p_{\mathcal{G}}(0) =  \sum_{r \ge 0} \sqrt{2 \over 3} \left(\frac{4}{3}\right)^r h_r {r! \over (2r)!}.
$$
Moreover, for each $\lambda$ we have
$$
    0<  p_{\mathcal{G}}(\lambda) < 1.
$$
\end{theorem}

\begin{proof}
If all the $H_i$ are 3-connected, then clearly a graph is in  $\mathcal{G}$ if and only its kernel is in $\mathcal{G}$. The probability $p_{\mathcal{G}}(\lambda)$ is then computed as in Theorem \ref{th:planar}. It is positive since $\mathcal{G}$ contains all trees and unicyclic graphs, which  contribute with positive probability (although tending to 0 as $\lambda \to\infty$). To prove that it is less than one, let $t$ be the largest size of the excluded minors $H_i$. By splitting vertices it is easy to construct cubic graphs containing $K_{t+1}$ as a minor, hence
$G(\lambda)$ contains $K_{t+1}$ as a minor with positive probability (alternatively, see the argument at the end of \cite{LuczakStruct}). It follows that $1 - p_{\mathcal{G}}(\lambda) >0 $.
\end{proof}

\begin{figure}[htb]
    \begin{center}
\includegraphics[scale=0.4]{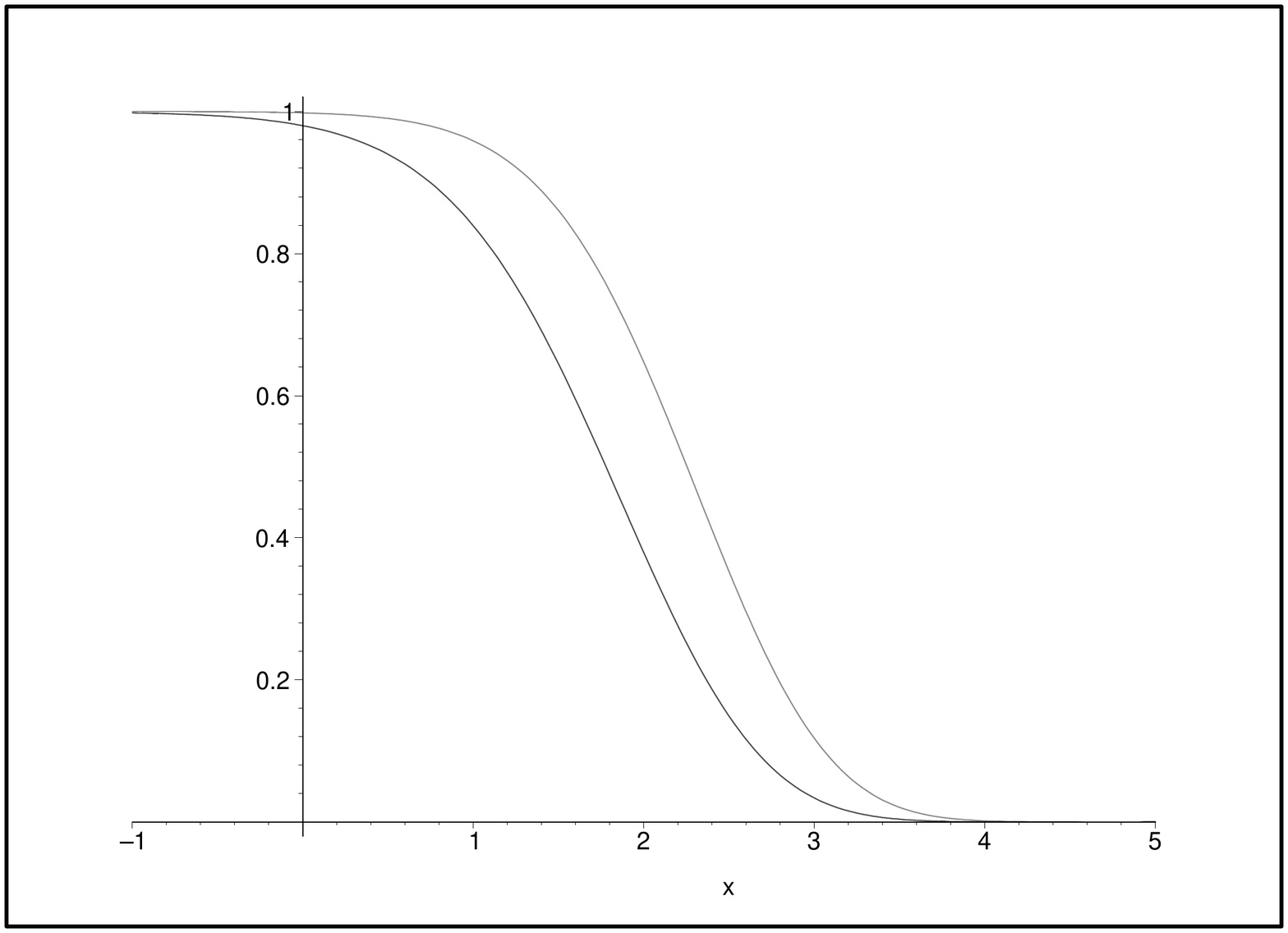}
\caption{The probability of $G(\lambda)$ being planar and of being series-parallel are both plotted for $\lambda \in [-1..4]$. The function on top corresponds to the planar case.}
    \end{center}
\end{figure}

In some cases of interest we are able to compute the numbers $h_r$ explicitly. Let $\mathcal{G} = \hbox{Ex}(K_4)$ be the class of series-parallel graphs.
The same system of equations as in Lemma \ref{le:cubic-planar} holds for series-parallel graphs with the difference that now $H(z) = 0$ (this is due to the fact that there are no 3-connected series-parallel graphs). The generating function for cubic series-parallel multigraphs can be computed as
$$
        G_{\rm sp}(z) = 1+{\frac {5}{24}}{z}^{2}+{\frac {337}{1152}}{z}^{4}+{\frac {55565}{
82944}}{z}^{6}+{\frac {15517345}{7962624}}{z}^{8}+ \cdots
$$
For instance, $[z^4](G(z)-G_{\rm sp}(z)) = {1 \over 24}$, corresponding to the fact that $K_4$ is the only cubic multigraph with $4$ vertices which is not series-parallel.
The limiting probability that $G(n,{n\over2})$ is series-parallel is
$$
    p_{\,\rm sp}(0) \approx  0.98003.
$$
See Figure 1 for a plot of $p_{\,\rm sp}(\lambda)$.

The class $\hbox{Ex}(K_4,K_{2,3})$ of outerplanar graphs does not fall directly under this scheme, since $K_{2,3}$ is not 3-connected, but adapting the equations in Lemma \ref{le:cubic-planar} (in particular the parallel decomposition encoded by $P(z)$) it is possible to enumerate exactly cubic outerplanar multigraphs. The first terms in the generating function are
$$
G_{\rm out}(z) =
1+{\frac {5}{24}}\,{z}^{2}+{\frac {337}{1152}}\,{z}^{4}+{\frac {55565}
{82944}}\,{z}^{6}+{\frac {14853793}{7962624}}\,{z}^{8} + \cdots
$$
The first discrepancy with $G_{\rm sp}(z)$ is at $z^8$, corresponding to the graph $K_{2,3}$ with either a loop or a double edge attached at the vertices of degree two.
The probability of being outerplanar is
$$
    p_{\,\rm out}(0) \approx  0.97979.
$$
We do not plot $p_{\,\rm out}(\lambda)$ in Figure 1 since it is too close to $p_{\,\rm sp}(\lambda)$ to see a clear distinction.

As another example, consider  excluding $K_{3,3}$. Since the only 3-connected non-planar graph in $\hbox{Ex}(K_{3,3})$ is $K_5$, which is not cubic, the limiting probability of being in this class is \textit{exactly} the same as of being planar, although $\hbox{Ex}(K_{3,3})$ is exponentially larger than the class of planar graphs \cite{k33}. But excluding the graph $K_{3,3}^+$, obtained by adding one edge to $K_{3,3}$, does increase the probability, since $K_{3,3}$ is in the class and is cubic and non-planar (the probability is computable since the 3-connected graphs in $\hbox{Ex}(K_{3,3}^+)$ are known \cite{k33}). Other classes such as $\hbox{Ex}(K_5 -e)$ or $\hbox{Ex}(K_3 \times K_2)$ can be analyzed too using the results from \cite{3-conn}.

It would be interesting to compute the probability that $G(\lambda)$ has genus $g$. For this we need to count cubic multigraphs of genus $g$ (orientable or not). We only know how to do this for $g=0$, the reason being that a 3-connected planar graph has a unique embedding in the sphere. This is not at all true in positive genus. It is true though that almost all 3-connected graphs of genus $g$ have a unique embedding in the surface of genus $g$ (see \cite{genus}). This could be the starting point for the enumeration, by counting first 3-connected maps of genus $g$ (a map is a graph equipped with a 2-cell embedding).
But this is not enough here, since we need the \emph{exact} numbers of graphs.

\section*{Acknowledgements}
The first and third authors acknowledge the warm hospitality and support from
the Laboratoire d'Informatique Algorithmique: Fondements et Applications (LIAFA) in Paris, where much of this work was done.

\end{document}